\documentclass[12pt, notitlepage]{amsart}
\usepackage{latexsym, amsfonts, amsmath, amssymb, amsthm}

\newtheorem{theorem}{Theorem}[section]
\newtheorem{lemma}[theorem]{Lemma}
\newtheorem{proposition}[theorem]{Proposition}

\newtheorem*{nonumtheorem}{Theorem}
\newtheorem{corollary}[theorem]{Corollary}

\newtheorem{definition}[theorem]{Definition}

\newcommand{\Z}{\mathbb{Z}}
\newcommand{\R}{\mathbb{R}}

\newcommand{\vol}{\lambda}
\usepackage{graphicx}
\usepackage{mathrsfs}

\begin{document}

\title{Large values of the additive energy in $\R^d$ and $\Z^d$}
\author{Xuancheng Shao}
\address{Department of Mathematics \\ Stanford University \\
450 Serra Mall, Bldg. 380\\ Stanford, CA 94305-2125}
\email{xshao@math.stanford.edu}

\begin{abstract}
Combining Freiman's theorem with Balog-Szemer\'{e}di-Gowers theorem
one can show that if an additive set has large additive energy, then
a large piece of the set is contained in a generalized arithmetic
progression of small rank and size. In this paper, we prove the
above statement with the optimal bound for the rank of the
progression. The proof strategy involves studying upper bounds for
additive energy of subsets of $\R^d$ and $\Z^d$.
\end{abstract}

\maketitle

\section{Introduction}

Let $G$ be an abelian group. For a finite set $A\subset G$, define
the sumset $A+A=\{a_1+a_2:a_1,a_2\in A\}$. The doubling constant of
$A$ is defined by $\sigma(A)=|A+A|/|A|$. A central topic in additive
combinatorics is to obtain structural information on $A$ when $A$
has small doubling. In this direction, Freiman's theorem asserts
that all such sets must lie inside a {\em generalized arithmetic
progression} (or simply a {\em progression} for short) of small rank
and size.

\begin{definition}[Progressions]
Let $r$ and $N_1,\cdots,N_r$ be positive integers. A progression of
rank $r$ is a set of the form
\begin{equation}\label{eq:GAP}
Q=\{x_0+x_1n_1+\cdots+x_rn_r:n_i\in\Z,0\leq n_i<N_i\}
\end{equation}
for some $x_0,x_1,\cdots,x_r\in G$. The progression $Q$ is said to
be proper if $|Q|=N_1N_2\cdots N_r$.
\end{definition}

\begin{theorem}[Freiman]
Let $A$ be a finite subset of a torsion-free abelian group $G$ with
$\sigma(A)=K$. Then there is a proper progression $Q$ containing $A$
with rank at most $C_1(K)$ and size at most $C_2(K)|A|$, where
$C_1(K)$ and $C_2(K)$ are constants depending only on $K$.
\end{theorem}

Chang \cite{chang}, building on earlier ideas of Ruzsa \cite{ruzsa},
obtained an effective version of Freiman's theorem with
$C_1(K)\leq\lfloor K-1+\epsilon\rfloor$ for any $\epsilon>0$ and
$|A|$ sufficiently large depending on $K$ and $\epsilon$, and
$C_2(K)\leq\exp(CK^2\log^3K)$ for some absolute constant $C$. We
will not worry too much about the quantity $C_2(K)$ here, but
instead focus on the quantity $C_1(K)$, the bound for the rank of
the progression $Q$. It turns out that Chang's bound on $C_1(K)$ is
optimal. This can be seen by taking
\begin{equation}\label{eq:lac}
A=\bigcup_{i=1}^{K-1}\{x_i+1,\cdots,x_i+N\}
\end{equation}
for a very lacunary sequence of integers $\{x_i\}$. This set $A$ has
doubling roughly $\sigma(A)\approx K$ for $N$ large enough, and can
only be covered by a progression of rank at least $K-1$.

On the other hand, observe that a typical progression of rank $K-1$
of the form \eqref{eq:GAP} should have doubling roughly $2^{K-1}$
rather than $K$, provided that $N_1,\cdots,N_{K-1}$ are large
enough. For the set $A$ given in \eqref{eq:lac}, instead of covering
it by a single progression of rank $K-1$, a more efficient way of
covering $A$ is to use $K-1$ copies of arithmetic progressions, or
progressions of rank $1$. A natural question arises as to whether
the bound for the rank of $Q$ can be improved if a few translates of
$Q$ are allowed. This question is answered by Freiman-Bilu
\cite{freiman-bilu}, and the following version of the statement is
due to Green-Tao \cite{green-tao}.

\begin{theorem}[Freiman-Bilu]
Let $A$ be a finite subset of a torsion-free abelian group $G$ with
$\sigma(A)=K$. Then for any $\epsilon>0$ there is a proper
progression $Q$ of rank at most $\lfloor\log_2K+\epsilon\rfloor$ and
size at most $|A|$, such that $A$ can be covered by $C(K,\epsilon)$
translates of $Q$ for some constant $C(K,\epsilon)$ depending on $K$
and $\epsilon$.
\end{theorem}

In fact, one can take $C(K,\epsilon)$ above to be
$\exp(CK^3\log^3K)/\epsilon^{CK}$ for some absolute constant $C$.
Although not the main concern of the current paper, it is a
difficult problem to obtain polynomial dependence on $K$ for this
quantity. This is related to the Polynomial Freiman-Ruzsa Conjecture
(PFR); see \cite{green-pfr} for the precise statement of this
conjecture (in the finite field setting).

Having explained the basic structure theorems for finite sets $A$ with small doubling, we now turn to another measurement of additive structure. The {\em additive energy} of $A$, denoted by $e(A)$, is defined by
\[ e(A)=\frac{\#\{(a_1,a_2,a_3,a_4)\in A^4:a_1+a_2=a_3+a_4\}}{|A|^3}. \]
Basically $e(A)$ counts the number of {\em additive quadruples} $(a_1,a_2,a_3,a_4)$ with $a_1+a_2=a_3+a_4$, and is normalized so that $e(A)\in [0,1]$. We will be concerned with sets having large additive energy.

A simple application of Cauchy-Schwarz inequality gives the estimate $e(A)\geq 1/\sigma(A)$. Hence small doubling implies large additive energy. In general, the converse of this statement is false: if one adds $o(|A|)$ random elements to $A$, then the additive energy $e(A)$ changes only by $o(1)$, while the doubling $\sigma(A)$ could change dramatically. However, the Balog-Szemer\'{e}di-Gowers theorem says that large additive energy implies small doubling for a large piece of $A$. This theorem has become an important tool in additive combinatorics; see Section 6.4 of \cite{tao-vu} for a proof and references to the original papers.

\begin{theorem}[Balog-Szemer\'{e}di-Gowers]
Let $A$ be a finite subset of an abelian group $G$ with $e(A)\geq 1/K$. Then there is a subset $A'\subset A$ such that $|A'|\geq C^{-1}K^{-C}|A|$ and $\sigma(A')\leq CK^C$ for some absolute constant $C>0$.
\end{theorem}

We now turn to the statement of our main result, which is a hybrid
of the Freiman-Bilu theorem and the Balog-Szemer\'{e}di-Gowers
theorem. Roughly speaking, our result asserts that, if $e(A)\geq
1/K$, then a large piece of $A$ is contained in a proper progression
$Q$ of small rank and size. This qualititative assertion is simply a
consequence of Freiman's theorem and the Balog-Szemer\'{e}di-Gowers
theorem. The main innovation of our result is an optimal dependence
of the rank of $Q$ on the constant $K$.

Before stating it, let us first discuss what the optimal dependence
of the rank of $Q$ on $K$ should be. If $I$ is an arithmetic
progression, then $e(I)=2/3+o(1)$ as $|I|\rightarrow\infty$. If $Q$
is a proper progression of the form \eqref{eq:GAP}, and if $Q+Q$ is
proper, a moment's thought confirms that $e(Q)$ is roughly $(2/3)^r$
when $N_1,\cdots,N_r$ are large. Note that a progression of rank $r$
is the image of a box under a homomorphism $\Z^r\rightarrow G$. It
turns out that if one uses a ball instead of a box, one can create
sets with even larger additive energy.

Consider the set $A\subset\Z^r$ consisting of lattice points inside
the ball of radius $R$ centered at the origin:
\begin{equation}\label{ball}
A=\{(x_1,\cdots,x_r):x_1^2+\cdots+x_r^2\leq R^2\}.
\end{equation}
 As
$R\rightarrow\infty$, the additive energy $e(A)$ approaches the
additive energy of the closed unit ball in $\R^r$, which we denote
by $e_r$ (although we have not defined the additive energy for
compact subsets of $\R^r$, the reader should not have trouble
figuring out what the definition should be; see Section 2 for the
precise definition). One can compute, for example, that $e_1=2/3$
and $e_2=1-16/3\pi^2\approx 0.46$. Note that $e_2>4/9$, showing that
this set $A\subset\Z^2$ indeed has larger additive energy than a
progression of rank $2$. A lengthy but standard computation
involving incomplete Bessel functions shows that
$e_r=(4\sqrt{3}/9+o(1))^r\approx 0.77^r$ for large $r$. The constant
$4\sqrt{3}/9$ here shows up as well in the sharp Hausdorff-Young
inequality; see Section \ref{sec:hy}.

\begin{nonumtheorem}[Theorem \ref{thm:bsgv}, simplified version]
Let $A$ be a finite subset of a torsion-free abelian group $G$ with
$e(A)\geq e_{r+1}+\epsilon$ for some $\epsilon>0$ and positive
integer $r$. Then there is a subset $A'\subset A$ with $|A'|\geq
c(r,\epsilon)|A|$ and a proper progression $Q$ containing $A'$ of
rank at most $r$ and size at most $C(r,\epsilon)|A'|$, where
$c(r,\epsilon)$ and $C(r,\epsilon)$ are constants depending only on
$r$ and $\epsilon$.
\end{nonumtheorem}

The dependencies of $c(r,\epsilon)$ and $C(r,\epsilon)$ on $r$ and
$\epsilon$ can be made explicit, although we did not care to do so
(one should get exponential dependence on $\epsilon$). The lower
bound for $e(A)$ in the hypothesis is sharp, which can be seen by
considering sets of the type \eqref{ball}.

In the next section, we will state the main result in a more general
form, describe the main ingredients in the proof of it, and give an
application of it to the carries problem.

\smallskip

{\bf Acknowledgments.} The author is grateful to Kannan
Soundararajan for many valuable discussions, and to Terry Tao for
pointing to him the references \cite{christ1,christ2}.


\section{Statement of results}

Let $A$ be a finite subset of an abelian group $G$. Recall that the
additive energy $e(A)$ counts the (normalized) number of solutions
to the linear equation $a_1+a_2-a_3-a_4=0$ with $a_i\in A$. We
consider a more general situation. For any positive integer $k\geq
3$ and finite subsets $A_1,\cdots,A_k\subset G$, define
\[ E_k(A_1,\cdots,A_k)=\#\{(a_1,\cdots,a_k)\in A_1\times\cdots\times A_k:a_1+\cdots+a_k=0\}. \]
In particular, we have $e(A)=E_4(A,A,-A,-A)/|A|^3$. When $A_1=\cdots=A_k=A$ we will simply write $E_k(A)$ for $E_k(A_1,\cdots,A_k)$.

In the qualitative aspect, studying $E_k(A)$ is not too much
different from studying the additive energy $e(A)$, due to the fact
that if $E_k(A)$ is large, then $e(A)$ is also large, and vice versa
(provided that $A=-A$). However, since we will be interested in the
quantitative aspect, our main result will be stated for general $k$.

We will also need to study additive energy for compact subsets of $\R^d$. In this setting, for compact subsets $A_1,\cdots,A_k\subset\R^d$, define
\[ E_k(A_1,\cdots,A_k)=\vol(\{(a_1,\cdots,a_{k-1})\in A_1\times\cdots\times A_{k-1}:-a_1-\cdots-a_{k-1}\in A_k\}),\]
where $\vol$ denotes the usual Lebesgue measure in $\R^{d(k-1)}$. As
before, when $A_1=\cdots=A_k=A$ we will simply write $E_k(A)$ for
$E_k(A_1,\cdots,A_k)$.

In the sequel we shall always use $\vol$ to denote Lebesgue measure in the appropriate dimension. We also remark that the notation $E_k$ is used in two different ways, both for finite sets in abelian groups and for compact sets in Euclidean spaces.

\begin{theorem}\label{thm:bsgv}
Let $d$ be a positive integer and $A_1,\cdots,A_k$ ($k\geq 3$) be
finite subsets of a torsion-free abelian group $G$. Let $X$ be the
union $A_1\cup\cdots\cup A_k$, and $B_i\subset\R^{d+1}$ ($1\leq
i\leq k$) be the closed ball centered at the origin with
$\vol(B_i)=|A_i|$. Suppose that
\[ E_k(A_1,\cdots,A_k)\geq E_k(B_1,\cdots,B_k)+\epsilon |X|^{k-1} \]
for some $\epsilon>0$. Then there is a proper progression $Q$ of
rank at most $d$ such that $|Q|\ll_{\epsilon}|X|$ and $|Q\cap
X|\gg_{\epsilon}|X|$.
\end{theorem}

Here and in the sequel, the implied constants in the $\ll$ or $\gg$
symbols are allowed to depend on the parameters $k,d,r$.
Dependencies on other parameters such as $\epsilon$ will be
explicitly mentioned.

\subsection{Outline of proof of Theorem \ref{thm:bsgv}}

To prove Theorem \ref{thm:bsgv}, it is necessary to understand how large $E_k(A_1,\cdots,A_k)$ can be, provided that the sets $A_i$ do not have large low-dimensional pieces. The following proposition is a result along such lines.

\begin{proposition}[Large additive energy in $\Z^d$]\label{prop:Zd}
Let $N_1,\cdots,N_d$ be non-negative integers and
$P=([-N_1,N_1]\times\cdots [-N_d,N_d])\cap\Z^d$ be a box. Let $k\geq
3$ and $A_1,\cdots,A_k\subset P$. Assume that $|A_1|\leq\cdots\leq
|A_k|$. Let $B_1,\cdots,B_k\subset\R^d$ be closed balls centered at
the origin with $\vol(B_i)=|A_i|$. Then
\[ E_k(A_1,\cdots,A_k)\leq E_k(B_1,\cdots,B_k)+O(|A_2|\cdots |A_{k-1}|\vol(\partial P)), \]
where $\vol(\partial P)=|P|\sum_{i=1}^d(2N_i+1)^{-1}$.
\end{proposition}

Comparing the error term above with the trivial upper bound
$|A_1||A_2|\cdots |A_{k-1}|$ for $E_k(A_1,\cdots,A_k)$, we see that
this result is nontrivial when $A_1$ is a dense subset of $P$ and
the side lengths $N_1,\cdots,N_d$ of $P$ are all sufficiently large
(depending on the density of $A_1$). Under these assumptions Proposition
\ref{prop:Zd} essentially says that $E_k(A_1,\cdots,A_k)$ is
maximized when each $A_i$ consists of lattice points inside a ball
in $\R^d$.

The deduction of Theorem \ref{thm:bsgv} from Proposition
\ref{prop:Zd} roughly goes as follows. A standard decomposition
theorem (see Proposition \ref{prop:structure} below) allows us to
assume that $X$ (the union of $A_i$) has small doubling. By
Freiman's theorem, $X$ is contained in a progression of small rank
and size, and is thus Freiman isomorphic to a dense subset of a box
in $\Z^d$ for some small $r$. If $r\geq d+1$ and all the side
lengths of this box are large, then Proposition \ref{prop:Zd} gives
an upper bound for $E_k(A_1,\cdots,A_k)$ which contradicts the
hypothesis in Theorem \ref{thm:bsgv}. Hence the box containing $X$
can have at most $d$ large side lengths. The conclusion of Theorem
\ref{thm:bsgv} easily follows from here.

Proposition \ref{prop:Zd} will follow from its continuous analogue, which is simpler to state and to prove.

\begin{proposition}[Large additive energy in $\R^d$]\label{prop:Rd}
Let $k\geq 3$ and Let
$A_1,\cdots,A_k\subset\R^d$ be compact subsets. Let
$B_i\subset\R^d$ be the closed ball centered at the origin with $\vol(B_i)=\vol(A_i)$. Then $E_k(A_1,\cdots,A_k)\leq
E_k(B_1,\cdots,B_k)$.
\end{proposition}

In words, additive energy is maximum when the sets are balls. The
case $d=1$ easily follows from a rearrangement theorem of
Hardy-Littlewood (see Lemma \ref{lem:hlg} below). For $d\geq 2$ the
result seems to be new. There is a striking resemblance between
Proposition \ref{prop:Rd} and the Brunn-Minkowski's inequality,
which asserts that
\[ \lambda(A_1+A_2)^{1/d}\geq\lambda(A_1)^{1/d}+\lambda(A_2)^{1/d} \]
for convex bodies $A_1,A_2\subset\R^d$, with the equality achieved
if and only if $A_1$ and $A_2$ are homothetic. In fact, our proof of
Proposition \ref{prop:Rd} is inspired by Blaschke's proof of the
Brunn-Minkowski inequality, which the author learned from the
excellent survey paper by Gardner \cite{gardner}.

\subsection{Connection with Hausdorff-Young
inequality}\label{sec:hy}

In this subsection, consider the situation when $A=-A$ is a
symmetric subset of $\R^d$ and $k$ is even. In this case, standard
Fourier analysis shows that $E_k(A)=\|\widehat{1_A}\|_k^k$, where
$1_A$ is the characteristic function of $A$. For a compactly
supported continuous function $f$ on $\R^d$, an upper bound for
$\|\hat{f}\|_q$ for any $q\geq 2$ is provided by the Hausdorff-Young
inequality:
\[ \|\hat{f}\|_q\leq C_d\|f\|_p, \]
where $p$ is the conjugate exponent of $q$ satisfying $1/p+1/q=1$.
Beckner \cite{beckner} obtained this inequality with the constant
$C_d=(p^{1/2p}q^{-1/2q})^d$, which is sharp when $f$ is Gaussian.

Approximating $1_A$ by continuous functions we conclude that, for
$q$ an even integer,
\[ E_q(A)=\|\widehat{1_A}\|_q^q\leq
C_d^q\|f\|_p^q=(p^{(q-1)/2}q^{-1/2})^d|A|^{q-1}.
\]
In particular, when $q=4$ and $p=4/3$ we get $E_4(A)\leq
(4\sqrt{3}/9)^d|A|^3$, and note that $E_4(A)$ is exactly the
(unnormalized) additive energy of $A$.

Thus Proposition \ref{prop:Rd} can be thought of as an improvement
of the constant $C_d$ in the Hausdorff-Young inequality when the
function $f$ is the characteristic function of a set, and this
improvement is more significant for small $d$. In fact, our method
actually produces a sharp inequality
\[ \|\hat{f}\|_q\leq C_d'\|f\|_{\infty}^{1/q}\|f\|_1^{1/p} \]
for $q$ an even integer and $f$ a compactly supported continuous
function, with equality achieved when $f=1_B$ for some closed ball
$B\subset\R^d$ centered at the origin. But we shall not need this
generalization here. For a related result concerning near
extremizers for the Hausdorff-Young inequality see
\cite{christ1,christ2}.

Finally, note that the $q=4$ case of the Hausdorff-Young inequality
gives a sharp upper bound for the Gowers $U^2$-norm of functions on
$\R^d$. In a recent work of Eisner and Tao \cite{eisner-tao}, this
is generalized to sharp upper bounds for Gowers $U^k$-norm of
functions on $\R^d$ for $k>2$. It is an interesting problem to
investigate sharp upper bounds for Gowers $U^k$-norm of compact sets
in $\R^d$ with fixed measure. Such estimates could have applications
in problems with more combinatorial nature.

\subsection{Application to the carries problem in $\Z^d$}

Let $G$ be an arbitrary group and $H\subset G$ be a finite-index
normal subgroup. Let $A$ be a set of coset representatives for $H$
in $G$, and consider the quantity
\[ c(A)=E_3(A,A,-A)/|A|^2, \]
which counts the number of solutions to $a_1+a_2=a_3$ with $a_i\in
A$, and normalized so that $c(A)\in [0,1]$. For $G=\Z$ and $H=b\Z$,
this is related to the number of carries occuring when two digits in
base $b$ are added. See \cite{carry} for a more detailed account of
this problem and various results. Using Theorem \ref{thm:bsgv} we
obtain a nontrivial upper bound for $c(A)$ when $G=\Z^d$ and
$H=(b\Z)^d$.

\begin{corollary}[Carries problem in $\Z^d$]\label{thm:carry}
Let $\epsilon>0$. Then for a sufficiently large positive integer $b$, and any
set $A$ of coset representatives for $(b\Z)^d\subset\Z^d$, we have $c(A)\leq c_d+\epsilon$, where $c_d=E_3(B,B,-B)/\vol(B)^2$ for a closed ball $B\subset\R^d$.
\end{corollary}

\begin{proof}
Suppose that $c(A)\geq c_d+\epsilon$. By Theorem \ref{thm:bsgv}
there is a $(d-1)$-dimensional proper progression $Q$ of the form
\[ Q=\{x_0+x_1n_1+\cdots+x_{d-1}n_{d-1}:n_i\in\Z,0\leq n_i<N_i\}  \]
for some $x_0,x_1,\cdots,x_{d-1}\in\Z^d$, such that
$|Q|=N_1N_2\cdots N_{d-1}\ll_{\epsilon}|A|$ and $|Q\cap
A|\gg_{\epsilon}|A|$.

Since $|A|=b^d$ and $|Q|\gg_{\epsilon}|A|$, we have
$N_j\gg_{\epsilon}b^{d/(d-1)}$ for some $1\leq j\leq d-1$. Fix such
a $j$. For any $r\pmod b$, let
\[ Q_r=\{x_0+x_1n_1+\cdots+x_{d-1}n_{d-1}:n_i\in\Z,0\leq n_i<N_i,n_j\equiv r\pmod b\}. \]
Then $|Q_r\cap A|\gg_{\epsilon}|A|/b=b^{d-1}$ for some $r$. On the
other hand, since $A$ consists of coset representatives for
$(b\Z)^d$ in $\Z^d$, for any fixed
$n_1,\cdots,n_{j-1},n_{j+1},\cdots,n_{d-1}$ there is at most one
value of $n_j$ such that $x_0+x_1n_1+\cdots+x_{d-1}n_{d-1}\in
Q_r\cap A$. Hence
\[ b^{d-1}\ll_{\epsilon}|Q_r\cap A|\leq \prod_{i\neq j}N_i=\frac{|Q|}{N_j}\ll_{\epsilon} b^{-d/(d-1)}|A|. \]
This is a contradiction for sufficiently large $b$.
\end{proof}

In particular, one can compute that $c_1=3/4$ and
$c_2=1-3\sqrt{3}/4\pi\approx 0.59$. We conjecture that $c(A)\leq
(3/4)^d+\epsilon$ is the optimal bound, achieved (for odd $b$) when
$A$ is the square box $[-(b-1)/2,(b-1)/2]^d$ centered at the origin.


\section{Large Additive Energy in $\R^d$}\label{sec:Rd}

In this section we prove Proposition \ref{prop:Rd}.

\subsection{The case $d=1$}

The $d=1$ case will follow from its discrete analogue due to Hardy
and Littlewood. See Theorem 376 of \cite{hardy-littlewood-polya},
and also \cite{gabriel,lev} for some related results.

\begin{lemma}[Hardy-Littlewood]\label{lem:hlg}
Let $k\geq 3$ and $A_1,\cdots,A_k\subset\Z$ be finite subsets with
$|A_i|$ odd ($1\leq i\leq k$). Let $I_i\subset\Z$ be the interval
centered at the origin with $|I_i|=|A_i|$. Then
$E_k(A_1,\cdots,A_k)\leq E_k(I_1,\cdots,I_k)$.
\end{lemma}

\begin{lemma}[$d=1$ case]\label{lem:d=1asym}
Let $k\geq 3$ and $A_1,\cdots,A_k\subset\R$ be compact subsets. Let
$I_i\subset\Z$ be the closed interval centered at the origin with
$\vol(I_i)=\vol(A_i)$. Then $E_k(A_1,\cdots,A_k)\leq
E_k(I_1,\cdots,I_k)$.
\end{lemma}

\begin{proof}
By standard measure theory, we may approximate each $A_i$ by a
finite disjoint union of intervals, and we may also assume that all
these intervals have rational endpoints. Let $M$ be a common
denominator of all these endpoints. By dividing intervals into
subintervals if necessary, we may assume that all these intervals
have length $1/M$. Finally, we may further assume that the number of
these intervals is odd. This leads to considering the case when
$A_i$ takes the form
\[ A_i=\bigcup_{j=1}^{2n_i+1}U_{ij}, \]
where $U_{ij}=[a_{ij},a_{ij}+1/M]$ for some positive integer $M$ and
rationals $a_{ij}$ with $Ma_{ij}\in\Z$, so that
$\vol(A_i)=(2n_i+1)/M$. Let $A_i'\subset\Z$ be the set
$\{Ma_{ij}:1\leq j\leq 2n_i+1\}$ and $I_i'$ be the interval
$I_i'=[-n_i,n_i]\cap\Z$.

Denote by $U$ the interval $[0,1/M]$. Note that
\[
E_k(A_1,\cdots,A_k)=\sum_{j_1,\cdots,j_k}E_k(U_{1j_1},\cdots,U_{kj_k})=\sum_{j_1,\cdots,j_k}E_k(U+a_{1j_1}+\cdots+a_{kj_k},U,\cdots,U).
\]
For any $s$ with $Ms\in\Z$, the number of ways to write
$s=a_{1j_1}+\cdots+a_{kj_k}$ for some $j_1,\cdots,j_k$ is equal to
$E_k(A_1'-Ms,A_2',\cdots,A_k')$. Hence
\[ E_k(A_1,\cdots,A_k)=\sum_sE_k(A_1'-Ms,A_2',\cdots,A_k')\cdot E_k(U+s,U,\cdots,U). \]
By Lemma \ref{lem:hlg}, we have $E_k(A_1'-Ms,A_2',\cdots,A_k')\leq
E_k(I_1',\cdots,I_k')$. It follows that
\begin{equation}\label{d=11}
E_k(A_1,\cdots,A_k)\leq
E_k(I_1',\cdots,I_k')\sum_sE_k(U+s,U,\cdots,U)=\frac{1}{M^{k-1}}E_k(I_1',\cdots,I_k').
\end{equation}

Now consider $E_k(I_1,\cdots,I_k)$, where $I_i$ is the interval
$[-(2n_i+1)/(2M),(2n_i+1)/(2M)]$. Denote by $V$ the interval
$[-1/2M,1/2M]$. A similar analysis as above shows that
\begin{align*}
E_k(I_1,\cdots,I_k) &=\sum_{a_1\in I_1'}\cdots\sum_{a_k\in
I_k'}E_k\left(V+\frac{a_1+\cdots+a_k}{M},V,\cdots,V\right)
\\
&=\sum_sE_k(I_1'-Ms,I_2',\cdots,I_k')\cdot E_k(V+s,V,\cdots,V).
\end{align*}
Note that $E_k(V+s,V,\cdots,V)$ vanishes unless $|s|\leq k/M$, and
for $|s|\leq k/M$,
\[ E_k(I_1'-Ms,I_2',\cdots,I_k')\geq E_k((I_1'-Ms)\cap
I_1',I_2',\cdots,I_k')\geq E_k(I_1',\cdots,I_k')-k|I_2'|\cdots
|I_{k-1}'|. \] Combining this with \eqref{d=11} we get
\begin{align*}
 E_k(I_1,\cdots,I_k) &\geq
[E_k(I_1',\cdots,I_k')-k|I_2'|\cdots
|I_{k-1}'|]\sum_sE_k(V+s,V,\cdots,V)
\\
&\geq
\frac{1}{M^{k-1}}E_k(I_1',\cdots,I_k')-\frac{k}{M^{k-1}}|I_2'|\cdots
|I_{k-1}'| \\
&\geq
E_k(A_1,\cdots,A_k)-\frac{k}{M}\lambda(A_2)\cdots\lambda(A_{k-1}).
\end{align*}
Letting $M\rightarrow\infty$ we get $E_k(I_1,\cdots,I_k)\geq
E_k(A_1,\cdots,A_k)$.
\end{proof}

\subsection{The general case}

The main tool in this section is the \textit{Steiner
symmetrization}. For any bounded subset $K\subset\R^d$ and nonzero
vector $u\in\R^d$, the \textit{Steiner symmetral} $S_uK$ of $K$ in
the direction $u$ is the set obtained from $K$ by sliding each of
its chords parallel to $u$ so that they are bisected by the
hyperplane $u^{\perp}$ and taking the union of the resulting chords.
In other words, for any $b\in u^{\perp}$, if we let $J_u(K;b)=(b+\R
u)\cap K$ and $\ell_u(K;b)=\lambda(J_u(K;b))$ then
\[ S_u(K)=\bigcup_{b\in u^{\perp}}\{b+xu:x\in I(\ell_u(K;b))\}, \]
where $I(\ell)$ denote the interval $[-\ell/2,\ell/2]$.

Basic properties of Steiner symmetrals can be found in (2.10.30) of
\cite{federer}. In particular, if $K$ is compact, then $S_u(K)$ is
also compact. The following lemma shows that Steiner symmetrization
increases $E_k$.

\begin{lemma}\label{lem:steiner}
For any compact subsets $K_1,\cdots,K_k\subset\R^d$ and any nonzero vector $u\in\R^d$, we have $E_k(K_1,\cdots,K_k)\leq E_k(S_u(K_1),\cdots,S_u(K_k))$.
\end{lemma}

\begin{proof}
By Lemma \ref{lem:d=1asym} we have, for any $b_1,\cdots,b_k\in
u^{\perp}$,
\[ E_k(J_u(K_1;b_1),\cdots,J_u(K_k;b_k))\leq
E_k(I(\ell_u(K_1;b_1)),\cdots,I(\ell_u(K_k;b_k))).
\]
Note that $E_k(K_1,\cdots,K_k)$ and $E_k(S_u(K_1),\cdots,S_u(K_k))$
are the integrals of the left side and the right side above,
respectively, over the region $b_1+\cdots+b_k=0$. The desired
inequality then follows immediately.
\end{proof}

A classical result in geometric measure theory states that any
compact subset $K\subset\R^d$ can be transformed arbitrarily close
to a ball using Steiner symmetrizations. More precisely, define the
Hausdorff distance $d_H$ between two compact subsets
$C,D\subset\R^d$ to be
\[ d_H(C,D)=\sup_{x\in\R^d}|\text{dist}(x,C)-\text{dist}(x,D)|, \]
where $\text{dist}(x,C)$ and $\text{dist}(x,D)$ are the distances
from $x$ to $C$ and $D$, respectively.

\begin{lemma}\label{lem:approx}
Given a compact subset $K\subset\R^d$, let $B\subset\R^d$ be the
ball centered at the origin with $\vol(B)=\vol(K)$. For any
$\epsilon>0$, there exists a finite sequence of nonzero vectors
$u_1,\cdots,u_n$ such that $d_H(B,S_{u_n}\cdots
S_{u_1}(K))<\epsilon$.
\end{lemma}

\begin{proof}
See (2.10.31) of \cite{federer}.
\end{proof}

For a ball $B\subset\R^d$ centered at the origin with radius $r$,
let $B(\epsilon)$ be the ball centered at the origin with radius
$r+\epsilon$. Note that if $d_H(K,B)<\epsilon$ for some compact $K$,
then $K\subset B(\epsilon)$.

\begin{proof}[Proof of Proposition \ref{prop:Rd}]
Fix any $\epsilon>0$. By Lemma \ref{lem:approx} there exists a
sequence of Steiner symmetrizations $S_{u_1},\cdots,S_{u_{n_1}}$
such that $d_H(S_{u_{n_1}}\cdots S_{u_1}(A_1),B_1)<\epsilon$. Again
by Lemma \ref{lem:approx} there exists a sequence of Steiner
symmetrizations $S_{u_{n_1+1}},\cdots,S_{u_{n_2}}$ such that
$d_H(S_{u_{n_2}}\cdots S_{u_1}(A_2),B_2)<\epsilon$. Repeating this
process $k$ times, we get a sequence of Steiner symmetrizations
$S_{u_1},\cdots,S_{u_{n_k}}$ such that
\[ d_H(S_{u_{n_i}}\cdots S_{u_1}(A_i),B_i)<\epsilon \]
for each $1\leq i\leq k$.

Let $A_i''=S_{u_{n_i}}\cdots S_{u_1}(A_i)$ and
$A_i'=S_{u_{n_k}}\cdots S_{u_1}(A_i)$. Since
$d_H(A_i'',B_i)<\epsilon$, we have $A_i''\subset B_i(\epsilon)$, and
thus $A_i'\subset B_i(\epsilon)$. By Lemma \ref{lem:steiner} we have
\[ E_k(A_1,\cdots,A_k)\leq E_k(A_1',\cdots,A_k'). \]
Write $E_k(A_1',\cdots,A_k')$ as the sum of $2^k$ terms, each of the
form $E_k(C_1,\cdots,C_k)$ where $C_i$ is either $A_i'\cap B_i$ or
$A_i'\setminus B_i$. The term with $C_i=A_i'\cap B_i$ for each
$1\leq i\leq k$ contributes at most $E_k(B_1,\cdots,B_k)$, and all
the other terms contribute at most $M\vol(A_i'\setminus B_i)$ for
some constant $M$ depending on $A_1,\cdots,A_k$. Hence
\[ E_k(A_1,\cdots,A_k)\leq E_k(A_1',\cdots,A_k')\leq
E_k(B_1,\cdots,B_k)+O_{A_1,\cdots,A_k}\left(\sum_{i=1}^k\vol(A_i'\setminus
B_i)\right). \] Since $A_i'\subset B_i(\epsilon)$,
$\vol(A_i'\setminus B_i)\rightarrow 0$ as $\epsilon\rightarrow 0$.
Letting $\epsilon\rightarrow 0$ we get $E_k(A_1,\cdots,A_k)\leq
E_k(B_1,\cdots,B_k)$ as desired.
\end{proof}


\section{Large Additive Energy in $\Z^d$}\label{sec:Zd}

In this section we prove Proposition \ref{prop:Zd}. The proof uses
the machinery of {\em compressions}, which seems to originate from
the work of Freiman \cite{freiman} and also appeared in
\cite{bollobas-leader,gardner-gronchi,green-tao}, though our precise
definitions of this term will differ slightly from those in the
literature. For a finite subset $A\subset\Z^d$ and any $1\leq i\leq
d$, define the $i$-compression $C_i(A)$ of $A$ as follows. Identify
$\Z^d$ with the direct product $\Z^{d-1}\times\Z$ via the
isomorphism $\sigma_i:\Z^d\rightarrow\Z^{d-1}\times\Z$ defined by
$\sigma_i(x_1,\cdots,x_d)=((x_1,\cdots,x_{i-1},x_{i+1},\cdots,x_d),x_i)$.
For any $b\in\Z^{d-1}$, let $J_i(A;b)=\{c\in\Z:(b,c)\in A\}$ and
write $\ell_i(A;b)=\# J_i(A;b)$.  For any positive integer $\ell$,
let $I(\ell)\subset\Z$ be the interval centered at the origin with
length $\ell$ if $\ell$ is odd, and length $\ell+1$ if $\ell$ is
even. Finally define
\[ C_i(A)=\bigcup_{\substack{b\in\Z^{d-1}\\ \ell_i(A;b)>0}}\{b\}\times I(\ell_i(A;b)). \]
The set $A$ is called $i$-compressed if $C_i(A)=A$, and is called a
down-set if it is $i$-compressed for each $1\leq i\leq d$.

This compression operation is the discrete analogue of Steiner
symmetrization used in Section \ref{sec:Rd}. In the proof of
Proposition \ref{prop:Rd} we applied a sequence of Steiner
symmetrizations to transform an arbitrary compact set to a ball.
Here we will apply a sequence of compressions to transform an
arbitrary finite set to a down-set, and then argue that down-sets
can be very well approximated by compact sets (in terms of
estimating $E_k$).

We first record a simple useful lemma.

\begin{lemma}\label{lem:onceforall}
Let $A\subset\Z^d$ be a finite subset. If $A$ is $i$-compressed for
some $1\leq i\leq d$, then $C_jA$ remains $i$-compressed for any
$1\leq j\leq d$.
\end{lemma}

\begin{proof}
Since the compressions $C_i$ and $C_j$ only change the $i$th and the $j$th coordinates and keep the remaining coordinates fixed, we may assume, for notational convenience, that $A\subset\Z^2$ and $(i,j)=(1,2)$. Take any point $(x,y)\in C_2A$. To show that $C_2A$ is $1$-compressed, it suffices to show that $(x',y)\in C_2A$ whenever $|x'|\leq |x|$. Since $C_2A$ is $2$-compressed, we know that $(x,y')\in C_2A$ whenever $|y'|\leq |y|$. Hence $C_2A$ contains at least $2|y|+1$ points in the line $\{(x,t):t\in\Z\}$, and so does $A$. Suppose that $A$ contains the $2|y|+1$ points $(x,t_1),\cdots,(x,t_{2|y|+1})$. Since $A$ is $1$-compressed, it also contains the $2|y|+1$ points $(x',t_1),\cdots,(x',t_{2|y|+1})$. It follows that $C_2A$ contains the point $(x',y)$, as desired.
\end{proof}

The following lemma shows that $i$-compression increases the additive energy. This is the discrete analogue of Lemma \ref{lem:steiner}.

\begin{lemma}\label{lem:compress}
Let $A_1,\cdots,A_k\subset \Z^d$ be as in the statement of
Proposition \ref{prop:Zd}. Then for each $1\leq i\leq d$, we have
$E_k(A_1,\cdots,A_k)\leq E_k(C_i(A_1),\cdots,C_i(A_k))$ and
$|C_i(A_j)|\leq |A_j|+(2N_i+1)^{-1}|P|$ ($1\leq j\leq k$).
\end{lemma}

\begin{proof}
For the upper bound on $|C_i(A_j)|$, note that
\[ |C_i(A_j)|-|A_j|\leq\#\{b\in\Z^{d-1}:\ell_i(A;b)>0\}\leq (2N_i+1)^{-1}|P|. \]
For the increase of $E_k$ after $i$-compression, note that by Lemma
\ref{lem:hlg}, for any $b_1,\cdots,b_k\in\Z^{d-1}$ we have
\[ E_k(J_i(A_1;b_1),\cdots,J_i(A_k;b_k))\leq
E_k(I(\ell_i(A_1;b_1)),\cdots,I(\ell_i(A_k;b_k))). \] Since
$E_k(A_1,\cdots,A_k)$ and $E_k(C_i(A_1),\cdots,C_i(A_k))$ are the
sums of the left side and the right side above, respectively, over
all $b_1,\cdots,b_k$ with $b_1+\cdots+b_k=0$, the desired inequality
follows immediately.
\end{proof}

The following lemma shows that any finite set can be transformed to a down-set by compressions.

\begin{lemma}\label{lem:down}
Let $A_1,\cdots,A_k\subset \Z^d$ be as in the statement of
Proposition \ref{prop:Zd}. Let $A_j'=C_1C_2\cdots C_d(A_j)$ ($1\leq
j\leq k$). Then $A_j'$ is a down-set and $|A_j'|\leq
|A_j|+\vol(\partial P)$. Moreover, $E_k(A_1,\cdots,A_k)\leq
E_k(A_1',\cdots,A_k')$.
\end{lemma}

\begin{proof}
The inequalities $|A_j'|\leq |A_j|+\vol(\partial P)$ and
$E_k(A_1,\cdots,A_k)\leq E_k(A_1',\cdots,A_k')$ follow from applying
Lemma \ref{lem:compress} $d$ times. To show that $A_j'$ is a
down-set, it suffices to show that $A_j'$ is $i$-compressed for each
$1\leq i\leq d$. This follows by applying Lemma \ref{lem:onceforall}
to the $i$-compressed set $C_iC_{i+1}\cdots C_d(A_j)$.
\end{proof}

For finite subsets $A_1,\cdots A_k\subset\Z^d$ ($k\geq 3$) and
$s\in\Z^d$, let $S_k(A_1,\cdots,A_k;s)$ be the number of solutions
to $a_1+\cdots+a_k=s$ with $a_i\in A_i$. In particular
$S_k(A_1,\cdots,A_k;0)=E_k(A_1,\cdots,A_k)$. The following lemma
shows that, if at least one of the sets $A_1,\cdots,A_k$ is a
down-set, then $S_k(A_1,\cdots,A_k;s)$ is close to
$E_k(A_1,\cdots,A_k)$ as long as the coordinates of $s$ are small.

\begin{lemma}\label{lem:boundary}
Let $A_1,\cdots,A_k\subset \Z^d$ be as in the statement of
Proposition \ref{prop:Zd}. Assume that $A_k$ is a down-set. Then for
any $s\in\Z^d$ we have $S_k(A_1,\cdots,A_k;s)\geq
E_k(A_1,\cdots,A_k)-O_s(|A_1|\cdots |A_{k-2}|\vol(\partial P))$.
\end{lemma}

\begin{proof}
Let $A_k'=A_k\cap (A_k-s)$ and $A_k''=A_k\setminus (A_k-s)$. Then
\begin{align*}
S_k(A_1,\cdots,A_k;s) & =E_k(A_1,\cdots,A_{k-1},A_k-s) \\
& \geq
E_k(A_1,\cdots,A_{k-1},A_k')=E_k(A_1,\cdots,A_{k-1},A_k)-E_k(A_1,\cdots,A_{k-1},A_k'')
\\
&\geq E_k(A_1,\cdots,A_k)-|A_k''||A_1|\cdots |A_{k-2}|.
\end{align*}
It thus suffices to show that $|A_k''|=O_s(\lambda(\partial P))$.

Since $A_k$ is $1$-compressed, it is of the form
\[ A_k=\bigcup_{b\in\Z^{d-1}}[-U(b),U(b)]\times\{b\} \]
for some $U:\Z^{d-1}\rightarrow\Z$. Since $A_k$ is a down-set, in
particular it is symmetric with respect to each coordinate plane.
Hence the function $U$ satisfies the relation
$U(x_2,\cdots,x_d)=U(|x_2|,\cdots,|x_d|)$.

Moreover, we claim that $U$ is decreasing in each coordinate in the
region $x_2,\cdots,x_d\geq 0$. In fact, if we have, for example,
$u_y=U(y,x_3,\cdots,x_d)<U(z,x_3,\cdots,x_d)=u_z$ for some
$x_3,\cdots,x_d\geq 0$ and $0\leq y\leq z$, then
\[ (u_z,y,x_3,\cdots,x_d)\notin A_k,\ \
(u_z,z,x_3,\cdots,x_d)\in A_k. \] This contradicts the fact that
$A_k$ is $2$-compressed.

To complete the proof, write $s=(u,t)$ with $u\in\Z$ and
$t\in\Z^{d-1}$. Then
\[ A_k''=\bigcup_{b\in\Z^{d-1}}([-U(b),U(b)]\setminus
[-U(b+t)-u,U(b+t)-u])\times\{b\}. \] Hence by the symmetry and the
monotonicity of $U$, we have
\[ |A_k''|\leq\sum_{\substack{b\in\Z^{d-1}\\
U(b)\geq 0}}[u+2(U(|b|)-U(|b|+|t|))]\ll
u(2N_1+1)^{-1}|P|+\sum_{\substack{b\in\Z^{d-1}\\
b\geq 0}}(U(b)-U(b+|t|)),
\]
where $b\geq 0$ means that $b$ has nonnegative coordinates, and
$|t|$ is the vector obtained by taking the absolute value of each
coordinate of $t$. For the telescoping sum on the right above, the
term $U(b)$ appears only when $b<|t|$. The number of such $b$ is
\[ \ll_{t}(2N_2+1)\cdots (2N_d+1)\sum_{i=2}^{d}(2N_i+1)^{-1}.
\]
Combining this with the trivial upper bound $U(b)\leq N_1$, we get
$|A_k''|\ll_s\vol(\partial P)$. This completes the proof.
\end{proof}

\begin{proof}[Proof of Proposition \ref{prop:Zd}]

We may assume that $|A_1|\geq\vol(\partial P)$; otherwise there is nothing to prove.
By Lemma \ref{lem:down} there exist down-sets $A_1',\cdots,A_k'$
such that
\[ E_k(A_1,\cdots,A_k)\leq E_k(A_1',\cdots,A_k'),\ \
|A_j'|\leq |A_j|+\vol(\partial P). \] For each $1\leq j\leq k$,
consider the compact set $K_j$ obtained by taking the union of
$R(a)$ ($a\in A_j'$), where $R(a)$ is the unit box centered around
$a$ (if $a=(a^{(1)},\cdots,a^{(d)})$ then
$R(a)=[a^{(1)}-1/2,a^{(1)}+1/2]\times\cdots\times
[a^{(d)}-1/2,a^{(d)}+1/2]$). Clearly $K_j$ has volume equal to
$|A_j'|$, and thus by Proposition \ref{prop:Rd},
\begin{equation}\label{Zd1}
E_k(K_1,\cdots,K_k)\leq E_k(B_1',\cdots,B_k'),
\end{equation}
 where $B_j'$ is the
closed ball centered at the origin with
$\vol(B_j')=\vol(K_j)=|A_j'|$. Since
$\vol(B_j')\leq\vol(B_j)+\lambda(\partial P)$, we have
$\vol(B_j')\leq 2\vol(B_j)$ and
\begin{equation}\label{Zd2}
E_k(B_1',\cdots,B_k')\leq E_k(B_1,\cdots,B_k)+O(|A_2|\cdots
|A_{k-1}|\vol(\partial P)).
\end{equation}

To relate $E_k(K_1,\cdots,K_k)$ with $E_k(A_1,\cdots,A_k)$, note that
\[ E_k(K_1,\cdots,K_k)=\sum_{a_1,\cdots,a_k\in A}E_k(R(a_1),\cdots,R(a_k)). \]
Since
$E_k(R(a_1),\cdots,R(a_k))=E_k(R(a_1+\cdots+a_k),R(0),\cdots,R(0))$,
we have
\[ E_k(K_1,\cdots,K_k)=\sum_{s\in\Z^d}E_k(R(s),R(0),\cdots,R(0))\cdot S_k(A_1,\cdots,A_k;s). \]
Note that $E(R(s),R(0),\cdots,R(0))$ vanishes unless the coordinates of $s$ are all bounded by $k$. For such $s\in\Z^d$, we have by Lemma \ref{lem:boundary},
\[ S_k(A_1,\cdots,A_k;s)\geq E_k(A_1,\cdots,A_k)-O(|A_1|\cdots |A_{k-2}|\vol(\partial P)). \]
Since $\sum_s E(R(s),R(0),\cdots,R(0))=1$, we conclude that
\[ E_k(K_1,\cdots,K_k)\geq E_k(A_1,\cdots,A_k)-O(|A_1|\cdots |A_{k-2}|\vol(\partial P)). \]
The proof is completed by combining this with \eqref{Zd1} and
\eqref{Zd2}.

\end{proof}


\section{Proof of Theorem \ref{thm:bsgv}}

\subsection{A structure decomposition for additive sets}

A technical ingredient in the proof of Theorem \ref{thm:bsgv} is to
decompose an arbitrary set $X$ into some ``additively structured''
parts $X_1,\cdots,X_m$ with small doubling plus a leftover part
$X_0$ in such a way that there is little ``additive communication''
between $X_i$ and $X_j$ for distinct $i$ and $j$. As a consequence,
most of solutions to additive equations in $X$ occur inside $X_i$
for some $1\leq i\leq m$.

\begin{proposition}[Structure theorem, Proposition 3.2 of \cite{green-sisask}]\label{prop:structure}
Let $G$ be an abelian group and $X\subset G$ be a finite subset. Let $\eta,\eta'>0$ be parameters. Then there is a decomposition of $X$ as a disjoint union $X_1\cup\cdots\cup X_m\cup X_0$ such that
\begin{enumerate}
\item (Components are large). $|X_i|\gg_{\eta}|X|$ for each $1\leq i\leq m$;
\item (Components are structured). $|X_i+X_i|\ll_{\eta,\eta'}|X_i|$ for each $1\leq i\leq m$;
\item (Distinct components do not communicate). $E_4(X_i,X_j,-X_i,-X_j)\leq\eta' |X_i|^{3/2}|X_j|^{3/2}$ whenever $1\leq i<j\leq m$;
\item (Noise term). $E_4(X_0,X,-X_0,-X)\leq\eta |X|^3$.
\end{enumerate}
\end{proposition}

We remark that in the bound $|X_i|\gg_{\eta}|X|$ the implied
constant depends only on $\eta$ but not $\eta'$. In particular, this
implies that the number of parts $m=O_{\eta}(1)$.

The following lemma captures the idea that $E_k(Y_1,\cdots,Y_k)$ is
small if two of these sets have little additive communication.

\begin{lemma}\label{lem:small}
Let $G$ be an abelian group and $Y_1,\cdots,Y_k$ ($k\geq 3$) be
finite subsets. Then
\[ E_k(Y_1,\cdots,Y_k)\leq |Y_4|\cdots
|Y_k||Y_3|^{1/2}E_4(Y_1,Y_2,-Y_1,-Y_2)^{1/2}. \]
\end{lemma}

\begin{proof}
For any $h\in G$, let $r(h)$ be the number of ways to write
$h=y_1+y_2$ with $y_1\in Y_1$ and $y_2\in Y_2$. We first bound the
number of solutions to $y_1+y_2+y_3=g$ with $y_i\in Y_i$ for any
fixed $g\in G$. This number is
\begin{align*}
\sum_{y_3\in Y_3}r(g-y_3) &\leq |Y_3|^{1/2}\left(\sum_{y_3\in
Y_3}r(g-y_3)^2\right)^{1/2} \\
&\leq |Y_3|^{1/2}\left(\sum_{h\in
G}r(h)^2\right)^{1/2}=|Y_3|^{1/2}E_4(Y_1,Y_2,-Y_1,-Y_2)^{1/2}.
\end{align*}
Now, to count the number of solutions to $y_1+\cdots+y_k=0$ with
$y_i\in Y_i$, first fix $y_4,\cdots,y_k$ and then count the number
of solutions to $y_1+y_2+y_3=g$ with $g=-y_4-\cdots-y_k$. In this
way we obtain the bound
\[ E_k(Y_1,\cdots,Y_k)\leq |Y_4|\cdots
|Y_k||Y_3|^{1/2}E_4(Y_1,Y_2,-Y_1,-Y_2)^{1/2}. \]

\end{proof}

\subsection{Proof of Theorem \ref{thm:bsgv} in the small doubling
case}\label{sec:smalldoubling}

Recall that $X=A_1\cup\cdots\cup A_k$. First consider the situation
when $|X+X|\ll_{\epsilon}|X|$. By Freiman's theorem, $X$ lies in a
proper progression $Q'$ of rank $r\ll_{\epsilon}1$ and size
$|Q'|\ll_{\epsilon}|X|$. Using Theorem 2.1 of \cite{green-proper}
and enlarging $Q'$ if necessary, we may further assume that $Q'$ is
centered, meaning that it takes the form
\[ Q'=\{x_1n_1+\cdots+x_rn_r:n_i\in\Z,|n_i|\leq N_i\}, \]
and that $Q'$ is $k$-proper, meaning that
\[ kQ'=\{x_1n_1+\cdots+x_rn_r:n_i\in\Z,|n_i|\leq kN_i\}
\]
is again a proper progression. Assume also that $r\geq d+1$, since
otherwise we are already done by taking $Q=Q'$.

Consider the map $\sigma:Q'\rightarrow\Z^r$ defined by
\[ \sigma(x_1n_1+\cdots+x_rn_r)=(n_1,\cdots,n_r). \]
Since $Q'$ is $k$-proper, this is well-defined and bijective, and
moreover it is a Freiman $k$-isomorphism, meaning that for any
$q_1,\cdots,q_k,q_1',\cdots,q_k'\in Q'$, we have
\[ q_1+\cdots+q_k=q_1'+\cdots+q_k'\Longleftrightarrow
\sigma(q_1)+\cdots+\sigma(q_k)=\sigma(q_1')+\cdots+\sigma(q_k'). \]
In particular, $q_1+\cdots+q_k=0$ if and only if
$\sigma(q_1)+\cdots+\sigma(q_k)=0$. Hence $E_k$ is stable under
$\sigma$.

Before applying Proposition \ref{prop:Zd}, we make a simple
reduction to the case $r=d+1$. Without loss of generality assume
that $N_1\geq\cdots\geq N_r$. Note that the image $\sigma(Q')$ lies
inside the box $([-N_1,N_1]\times\cdots\times [-N_r,N_r])\cap\Z^r$,
and clearly there is a Freiman $k$-isomorphism
\[ \tau:([-N_1,N_1]\times\cdots\times [-N_r,N_r])\cap\Z^r\rightarrow
P, \] where $P$ is a box $[-N_1,N_1]\times\cdots [-N_d,N_d]\times
[-N_{d+1}',N_{d+1}']$ in $\Z^{d+1}$, with $N_{d+1}'\ll N_{d+1}\cdots
N_r$. Now consider the image of $A_i\subset X\subset Q'$ under the
Freiman $k$-isomorphism $\tau\circ\sigma$ and apply Proposition
\ref{prop:Zd}:
\[ E_k(A_1,\cdots,A_k)=E_k(\tau(\sigma(A_1)),\cdots,\tau(\sigma(A_k)))\leq
E_k(B_1,\cdots,B_k)+O(|A_2|\cdots |A_{k-1}|\vol(\partial P)),
\] where $B_i\subset\R^{d+1}$ is the closed ball centered at the
origin with $\vol(B_i)=|A_i|$.

Combining this with the hypothesis
\[ E_k(A_1,\cdots,A_k)\geq E_k(B_1,\cdots,B_k)+\epsilon |X|^{k-1},
\]
we conclude that $\vol(\partial P)\gg |X|$ and thus $\vol(\partial
P)\gg_{\epsilon}|P|$ (because $|P|\ll |Q'|\ll_{\epsilon}|X|$). Hence
at least one of the side lengths of $P$ is $O_{\epsilon}(1)$. It
follows that either $N_i\ll_{\epsilon}1$ for some $1\leq i\leq d$ or
$N_{d+1}'\ll_{\epsilon}1$. Recall that the $N_i$'s are arranged in
decreasing order and $r\ll_{\epsilon}1$. Hence in either case we
have $N_{d+1}\cdots N_r\ll_{\epsilon}1$.

To finish the proof (in the case when $X$ has small doubling), note
that $Q'$ is the union of $O(N_{d+1}\cdots N_r)=O_{\epsilon}(1)$
copies of proper progressions of rank at most $d$. Hence one of
those progressions $Q$ satisfies $|Q\cap X|\gg_{\epsilon}|X|$, as
desired.

\subsection{Proof of Theorem \ref{thm:bsgv} in the general case}

Now consider the general case when $X$ does not necessarily have
small doubling. Let $\eta=\eta(\epsilon)>0$ be a small parameter and
$\eta'=\eta'(\epsilon,\eta)>0$ be a smaller parameter. Apply
Proposition \ref{prop:structure} to the set $X$ with these
parameters to obtain a decomposition $X=X_1\cup\cdots X_m\cup X_0$
satisfying the listed properties. In particular, for $1\leq j\leq m$
we have $|X_j+X_j|\ll_{\eta,\eta'}|X_j|$ and $|X_j|\gg_{\eta}|X|$.
Moreover $m=O_{\eta}(1)$.

For $1\leq i\leq k$ and $0\leq j\leq m$, let $A_{ij}=A_i\cap X_j$.
Note that
\[ E_k(A_1,\cdots A_k)=\sum_{0\leq j_1,\cdots,j_k\leq m}E_k(A_{1j_1},\cdots,A_{kj_k}). \]
Split this sum into three parts $S_1,S_2,S_3$. Here $S_1$ is the contribution from the terms $j_1=\cdots=j_k>0$, $S_2$ is the contribution from the terms where one of $j_1,\cdots,j_k$ is zero, and $S_3$ is the contribution from the remaining terms, with $j_1,\cdots,j_k$ all positive but not all equal.

We show using Lemma \ref{lem:small} that the contributions from
$S_2$ and $S_3$ are negligible. For $S_2$, we have
\[ S_2\leq kE_k(X_0,X,\cdots,X)\ll
E_4(X_0,X,-X_0,-X)^{1/2}|X|^{k-5/2}\leq\eta^{1/2}|X|^{k-1}\leq\frac{1}{3}\epsilon
|X|^{k-1}, \] provided that $\eta$ is small enough depending on
$\epsilon$. For $S_3$, note that each summand appearing in $S_3$ is
bounded by $\eta'^{1/2}|X|^{k-1}$, and thus
\[ S_3\leq m^k\eta'^{1/2}|X|^{k-1}\leq\frac{1}{3}\epsilon |X|^{k-1}, \]
provided that $\eta'$ is small enough depending on $\eta$ and
$\epsilon$ (recall that $m=O_{\eta}(1)$).

It follows that
\[ S_1=\sum_{j=1}^mE_k(A_{1j},\cdots,A_{kj})\geq E_k(A_1,\cdots,A_k)-\frac{2}{3}\epsilon |X|^{k-1}\geq E_k(B_1,\cdots,B_k)+\frac{1}{3}\epsilon |X|^{k-1}. \]
We need the following simple lemma.

\begin{lemma}
Let $C_i,D_i\subset\R^d$  ($1\leq i\leq k$) be compact subsets. Let
$B_i\subset\R^d$ be the closed ball centered at the origin with
$\vol(B_i)=\vol(C_i)+\vol(D_i)$. Then
\[ E_k(B_1,\cdots,B_k)\geq E_k(C_1,\cdots,C_k)+E_k(D_1,\cdots,D_k). \]
\end{lemma}

\begin{proof}
Choose $t_1,\cdots,t_k\in\R^d$ with $t_1+\cdots+t_k=0$ such that
$C_i\cap (D_i+t_i)=\emptyset$ for each $1\leq i\leq k$. Note that
\begin{equation}\label{simple1}
E_k(D_1,\cdots,D_k)=E_k(D_1+t_1,\cdots,D_k+t_k).
\end{equation}
 By the
disjointness of $C_i$ and $D_i+t_i$, we have
\begin{equation}\label{simple2}
E_k(C_1,\cdots,C_k)+E_k(D_1+t_1,\cdots,D_k+t_k)\leq E_k(C_1\cup
(D_1+t_1),\cdots,C_k\cup (D_k+t_k)).
\end{equation}
 Since $\vol(C_i\cup
(D_i+r_i))=\vol(C_i)+\vol(D_i)=\vol(B_i)$, Lemma \ref{lem:d=1asym}
implies that
\[  E_k(C_1\cup (D_1+t_1),\cdots,C_k\cup (D_k+t_k))\leq E_k(B_1,\cdots,B_k). \]
The proof is completed by combining this with \eqref{simple1} and
\eqref{simple2}.
\end{proof}

Continuing with the proof of Theorem \ref{thm:bsgv}, we apply the
above lemma repeatedly to get
\[ E_k(B_1,\cdots,B_k)\geq\sum_{j=1}^mE_k(B_{1j},\cdots,B_{kj}), \]
where $B_{ij}\subset\R^{d+1}$ is the closed ball centered at the
origin with $\vol(B_{ij})=|A_{ij}|$. Combining this with the lower
bound for $S_1$ above, we conclude that there exists $1\leq j\leq m$
such that
\[ E_k(A_{1j},\cdots,A_{kj})\geq
E_k(B_{1j},\cdots,B_{kj})+\frac{\epsilon}{3m}|X|^{k-1}. \]

Since $|X_j+X_j|\ll_{\epsilon} |X_j|$, the proof in Section
\ref{sec:smalldoubling} shows that there is a proper progression $Q$
of rank at most $d$ such that $|Q|\ll_{\epsilon}|X_j|$ and $|Q\cap
X_j|\gg_{\epsilon}|X_j|$. We thus conclude that
\[ |Q|\ll_{\epsilon}|X|,\ \ |Q\cap X|\geq |Q\cap
X_j|\gg_{\epsilon}|X_j|\gg_{\epsilon}|X|, \] as desired.

\bibliography{note}{}
\bibliographystyle{plain}

\end{document}